\documentclass[a4paper,reqno]{amsart}

\usepackage{geometry}
\title{Diffeological Spaces with a Non-Smooth Derivation}
\author{Masaki Taho}                                                                               

\makeatletter
\@namedef{subjclassname@2020}{\textup{2020} Mathematics Subject Classification}
\makeatother
\subjclass[2020]{Primary~57P05; Secondary~58A05}

\keywords{diffeology, tangent space}

\address{GRADUATE SCHOOL OF MATHEMATICAL SCIENCES, THE
UNIVERSITY OF TOKYO, 3-8-1 KOMABA, MEGURO-KU, TOKYO, 153-8914,
JAPAN}
\email{taho@ms.u-tokyo.ac.jp}

\usepackage{amsmath} 
\usepackage{amssymb} 
\usepackage{amsfonts} 
\usepackage{mathrsfs} 
\usepackage{amsthm} 
\usepackage{bm} 
\usepackage[new]{old-arrows} 
\usepackage[dvipdfmx]{graphicx} 
\usepackage{wrapfig} 
\usepackage{color} 
\usepackage[labelsep=colon]{caption} 
\usepackage[labelformat=empty,subrefformat=parens]{subcaption} 

\usepackage{spalign}

\usepackage{multicol} 
\usepackage{units} 
\usepackage[all]{xy} 
\usepackage{tikz-cd} 
\usepackage[backref]{hyperref} 
\hypersetup{
    colorlinks=true,
    linkcolor= [rgb]{0.3,0.55,0.25},
    citecolor=[rgb]{0.164,0.39,0.77},
    urlcolor=[rgb]{0,0,1}
}

\usepackage[nameinlink,capitalize,noabbrev]{cleveref}

\crefname{thm}{theorem}{theorems}
\Crefname{thm}{Theorem}{Theorems}
\crefname{prop}{proposition}{propositions}
\Crefname{prop}{Proposition}{Propositions}
\crefname{lemma}{lemma}{lemmas}
\Crefname{lemma}{Lemma}{Lemmas}
\crefname{cor}{corollary}{corollaries}
\Crefname{cor}{Corollary}{Corollaries}
\crefname{defi}{definition}{definitions}
\Crefname{defi}{Definition}{Definitions}
\crefname{ex}{example}{examples}
\Crefname{ex}{Example}{Examples}
\crefname{remark}{remark}{remarks}
\Crefname{remark}{Remark}{Remarks}
\crefname{q}{question}{questions}
\Crefname{q}{Question}{Questions}
\crefname{conj}{conjecture}{conjectures}
\Crefname{conj}{Conjecture}{Conjectures}
\crefname{fact}{fact}{facts}
\Crefname{fact}{Fact}{Facts}

\crefname{equation}{equation}{equations}
\Crefname{equation}{Equation}{Equations}
\crefname{figure}{figure}{figures}
\Crefname{figure}{Figure}{Figures}
\crefname{table}{table}{tables}
\Crefname{table}{Table}{Tables}
\crefname{section}{section}{sections}
\Crefname{section}{Section}{Sections}

\usepackage{thmtools}
\usepackage{thm-restate}  

\newtheorem{thm}{Theorem}[section]
\newtheorem{prop}[thm]{Proposition}

\newtheorem*{mthm*}{Main Theorem}
\newtheorem*{thm*}{Theorem}
\newtheorem*{prop*}{Proposition}

\theoremstyle{definition}
\newtheorem{defi}[thm]{Definition}
\newtheorem{ex}[thm]{Example}
\newtheorem{remark}[thm]{Remark}

\newcommand{\dflg}{{\mathrm {Dflg}}}                 

\newcommand{\plotdom}[1]{U_{#1}}                      

\newcommand{\ad}{\operatorname{ad}}              

\newcommand{\id}{\operatorname{id}}

\newcommand{\colim}{\operatorname{colim}}

\newcommand{\germ}{G}                                

\newcommand{\righttangent}[2]{\hat{T}^R_{#1} (#2)}

\newcommand{\external}[2]{\hat{T}_{#1} (#2)}

\newcommand{\Hom}{\mathrm{Hom}}

\begin{document}

\begin{abstract}
    We show that on certain diffeological spaces there exist linear derivations that satisfy the Leibniz rule but are not smooth with respect to the given diffeology.
    This reveals that the notion of tangent space defined via all such derivations is strictly larger than the one defined using only smooth derivations, showing that smoothness cannot be recovered from the Leibniz rule alone.
\end{abstract}

\maketitle

\section{Introduction}
To study smooth structures beyond manifolds, Souriau proposed the concept of ``diffeologies'' in his 1980 paper \cite{Souriau}.
This led to the notion of diffeological spaces, extending the category of smooth manifolds to include quotients, mapping spaces,
and other non-manifold objects while preserving a reasonable concept of smooth maps; see, e.g., Baez--Hoffnung~\cite{Baez}.

Several notions of tangent spaces for diffeological spaces have been introduced by Iglesias-Zemmour~\cite{IZ}, Vincent~\cite{vincent}, and Christensen-Wu~\cite{CW}.  
This paper focuses on the \emph{external tangent space} $\external{x}{X}$ introduced in~\cite{CW}, 
which is the vector space of smooth derivations
\[
D \colon \germ(X,x)\to \mathbb{R},
\]
i.e., smooth $\mathbb{R}$-linear maps satisfying the Leibniz rule.
The version without the smoothness requirement, called the \emph{right tangent space} $\righttangent{x}{X}$, was studied in~\cite{taho}, 
where the corresponding functor was shown to coincide with the right Kan extension of the classical tangent functor.   
However, as pointed out in~\cite[3.4.2]{CW} and~\cite{taho}, 
it remained unknown whether these two constructions can differ for some diffeological space.

In this paper we show that they do differ for certain diffeological spaces. 
A concrete example is given by the bouquet of countably many copies of~$\mathbb{R}$. 
At the wedge point, the external tangent space of smooth derivations is countably infinite-dimensional, 
while the right tangent space is uncountably infinite-dimensional (throughout the paper, $\mathbb{N}$ denotes the set of positive integers).

\begin{restatable*}{thm}{maintheoremone}\label{thm:mainone}
  Let $X=\bigvee_{n\in\mathbb{N}} \mathbb{R}$ be the bouquet of countably many copies of~$\mathbb{R}$ equipped with the quotient diffeology, 
  and let $0\in X$ be the wedge point.  
  Then we have isomorphisms of vector spaces
  \[
  \external{0}{X}\cong \bigoplus_{n\in\mathbb{N}} \mathbb{R},
  \qquad
  \righttangent{0}{X}\cong \Hom_{\mathbb{R}}\!\left(\prod_{n\in\mathbb{N}}\mathbb{R},\mathbb{R}\right).
  \]
  In particular,
  \[
  \dim_{\mathbb{R}} \external{0}{X} = \aleph_0
  \qquad\text{and}\qquad
  \dim_{\mathbb{R}} \righttangent{0}{X} > \aleph_0.
  \]
  Hence the external and right tangent spaces do not coincide.
\end{restatable*}
  
This suggests that smoothness is a genuinely additional requirement. A derivation satisfying the Leibniz rule need not be smooth.


\section{Background on diffeological spaces and tangent spaces}\label{section:diffeology}

We briefly recall the notions of diffeological spaces and two constructions of tangent spaces for them.
For further details, see \cite{IZ}, \cite{CW}, \cite{vincent}, and \cite{taho}.

\subsection{Diffeological spaces}
\begin{defi}[{\cite[1.5]{IZ}}] \label{definition diffeology}
    Let $X$ be a set. A \emph{parametrization} of $X$ is a map $U\to X$ where $U$ is an open subset of $\mathbb{R}^n$ for some $n\in\mathbb{Z}_{\geq 0}$.
    A \emph{diffeology} on $X$ is a set of parametrizations $\mathscr{D}_X$ (we call an element of $\mathscr{D}_X$ a plot) such that the following three axioms are satisfied:
    \begin{description}
      \item[Covering] Every constant parametrization $U\to X$ is a plot.
      \item[Locality] Let $p\colon U\to X$ be a parametrization. If there is an open covering $\{U_{\alpha}\}$ of $U$ such that ${p|}_{U_{\alpha}}\in \mathscr{D}_X$ for all $\alpha$, then $p$ itself is a plot.  
      \item[Smooth compatibility] For every plot $p\colon U\to X$, every open set $V$ in $\mathbb{R}^m$ and $f\colon V\to U$ that is smooth as a map between Euclidean spaces, $p\circ f\colon V\to X$ is also a plot.
    \end{description}
    A diffeological space is a set equipped with a diffeology.
    We usually write $X$ for the diffeological space $(X,\mathscr{D}_X)$.
\end{defi}

\begin{defi}[{\cite[1.14]{IZ}}]
A map $f\colon X\to Y$ between diffeological spaces is \emph{smooth} if for every plot $p\colon U\to X$, the composition $f\circ p$ is a plot of $Y$.
\end{defi}

Diffeological spaces and smooth maps form a category $\dflg$ that contains smooth manifolds as a full subcategory.

\begin{ex}
If $M$ is a smooth manifold, the set of all smooth maps $U\to M$ from Euclidean open sets $U$ forms a diffeology, called the \emph{standard diffeology} on $M$.
\end{ex}

\begin{defi}[{\cite[2.8]{IZ}}]\label{def:D-topology}
  Let $X$ be a diffeological space. The \emph{$D$-topology} on $X$
  is the finest topology that makes all plots $p\colon U\to X$ continuous.
\end{defi}  

\begin{defi}[{\cite[1.55]{IZ}}] \label{def:product}
  Let $\{X_{\alpha}\}_{\alpha\in A}$ be a family of diffeological spaces. 
  The \emph{product diffeology} on $\prod_{\alpha\in A} X_{\alpha}$ is the diffeology whose plots are those maps $p\colon U\to\prod_{\alpha}X_{\alpha}$ for which $p_{\alpha}\circ p$ is a plot of $X_{\alpha}$ for every $\alpha\in A$, 
  where $p_{\alpha}\colon \prod_{\alpha\in A} X_{\alpha}\to X_{\alpha}$ is the canonical projection.
\end{defi}

\begin{defi}[{\cite[1.50]{IZ}}]\label{def:quotient}
Let $X$ be a diffeological space and $\sim$ an equivalence relation on $X$.  
The \emph{quotient diffeology} on $Y=X/\!\sim$ is the smallest diffeology (with respect to inclusion of plots) for which the quotient map $\pi\colon X\to Y$ is smooth.  
Equivalently, a parametrization $p\colon U\to Y$ is a plot if and only if each point of $U$ has a neighborhood $U'$ such that $p|_{U'}=\pi\circ q$ for some plot $q\colon U'\to X$.
\end{defi}

\begin{defi}[{\cite[1.39]{IZ}}]\label{def:coproduct}
  Let $\{X_i\}_{i\in I}$ be a family of diffeological spaces.  
  The \emph{coproduct diffeology} on the disjoint union $\bigsqcup_{i\in I} X_i$
  is the finest diffeology that makes each inclusion 
  \[
    \iota_i\colon X_i \hookrightarrow \bigsqcup_{i\in I} X_i
  \]
  smooth.  
  Equivalently, a map $p\colon U\to \bigsqcup_{i\in I} X_i$ is a plot
  if and only if for every $u\in U$ there exists a neighborhood $U'$ of $u$
  and an index $i\in I$ such that $p(U')\subset X_i$ and $p|_{U'}$ is a plot of $X_i$.
\end{defi}

\begin{defi}[{\cite[1.57]{IZ}}] \label{definition:functional}
  Let $X$ and $Y$ be diffeological spaces.  
  The \emph{functional diffeology} on $C^{\infty}(X,Y)$ is defined as follows: 
  a parametrization $p\colon \plotdom{p}\to C^{\infty}(X, Y)$ is a plot if and only if the map
  \[
    \ad(p)\colon \plotdom{p}\times X\to Y;(u,x)\mapsto p(u)(x)
  \]
  is smooth with respect to the product diffeology on $\plotdom{p}\times X$.
\end{defi}

\begin{ex}[Irrational tori, {\cite[Exercise~4]{IZ}}]\label{ex:irrational-torus}
  Let $\alpha\in\mathbb{R}\setminus\mathbb{Q}$.  
  The \emph{irrational torus} of slope $\alpha$ is the quotient diffeological space
  \[
    T_\alpha=\mathbb{R}/(\mathbb{Z}+\alpha\mathbb{Z})
  \]
  equipped with the quotient diffeology induced from $\mathbb{R}$.  
\end{ex}

\begin{ex}[Bouquet of lines]\label{ex:bouquet}
  Let $\{\mathbb{R}_n\}_{n\in\mathbb{N}}$ be countably many copies of $\mathbb{R}$,
  and define
  \[
    \bigvee_{n\in\mathbb{N}}\mathbb{R}_n
    =\left(\bigsqcup_{n\in\mathbb{N}}\mathbb{R}_n\right)\big/\bigl(0_n \sim 0_m \text{ for all } n,m\,\bigr),
  \]
  equipped with the quotient diffeology induced from the coproduct $\bigsqcup_{n\in\mathbb{N}}\mathbb{R}_n$ in \Cref{def:coproduct}. 
\end{ex}


\begin{ex}[Germ algebras, {\cite[paragraph before Definition~3.10]{CW}}]\label{ex:germ-algebra}
  Let $(X,x)$ be a based diffeological space.   
  The set of germs of smooth functions at~$x$ is defined as the colimit
  \[
    \germ(X,x) = \colim_{x\in B\subset X} C^{\infty}(B,\mathbb{R}),
  \]
  where the colimit ranges over all $D$-open neighborhoods $B$ of $x$.  
  The space $\germ(X,x)$ carries the quotient diffeology induced from the coproduct of the functional diffeologies on the spaces $C^{\infty}(B,\mathbb{R})$,  
  and becomes a diffeological $\mathbb{R}$-algebra under the pointwise operations.
\end{ex}

\medskip

\subsection{The external tangent spaces and the right tangent spaces}\label{subsec:external-right}

We next introduce the \emph{external tangent space} and the \emph{right tangent space}. 

\begin{defi} \label{def:external-right}
  Let $(X,x)$ be a based diffeological space and let $\germ(X,x)$ be the diffeological $\mathbb{R}$-algebra of germs of smooth functions at $x$ (\Cref{ex:germ-algebra}).
  A map $D\colon\germ(X,x)\to\mathbb{R}$ is said to be a \emph{right tangent vector} if it is linear and satisfies the Leibniz rule:
  \[
    D(gh)=g(x)D(h)+h(x)D(g)\quad \text{for all } g, h\in \germ(X,x). 
  \]
  The set of all right tangent vectors, denoted by $\righttangent{x}{X}$, forms a vector space called the \emph{right tangent space of $X$ at $x$}.

  An element $D\in\righttangent{x}{X}$ is further called an \emph{external tangent vector} if $D$ is smooth as a map of diffeological spaces.  
  The subspace of all external tangent vectors, denoted by $\external{x}{X}$, is called the \emph{external tangent space of $X$ at $x$} \cite[Definition~3.10]{CW}.
\end{defi}

\begin{remark}
  The right tangent space was introduced in~\cite{taho}, where the corresponding functor was shown to admit a characterization as the right Kan extension of the classical tangent functor.
  It was left open, however, whether the right and external tangent spaces always coincide.
  Already in~\cite[3.4.2]{CW}, it was remarked that it was unclear whether the smoothness of derivations follows automatically from the Leibniz rule.
\end{remark}

\begin{prop}[{\cite[Proposition 3.11]{CW}}]\label{prop:external-right-Ix}
  Let $(X,x)$ be a based diffeological space, and let 
  \[
    I_x(X)=\{f\in\germ(X,x)\mid f(x)=0\}
  \]
  be the kernel of the evaluation map $\germ(X,x)\to\mathbb{R}$ equipped with the subdiffeology from $\germ(X,x)$.
  \begin{enumerate}
    \item The following isomorphism holds: 
    \[
      \external{x}{X}\cong L^{\infty}(I_x(X)/I_x^2(X),\mathbb{R})=\{F\colon I_x(X)/I_x^2(X)\to\mathbb{R}\mid F \text{ is smooth and linear}\,\}.
    \]
    The isomorphism is given by
    \[
      \begin{aligned}
        \alpha&\colon \external{x}{X}\to L^{\infty}(I_x(X)/I_x^2(X),\mathbb{R}), & \alpha(D)([f])&=D([f]),\\
        \beta&\colon L^{\infty}(I_x(X)/I_x^2(X),\mathbb{R})\to \external{x}{X}, & \beta(F)([f])&=F([f]-[f(x)]).
      \end{aligned}
    \]
    \item The same construction, without the smoothness condition, yields a natural isomorphism
    \[
      \righttangent{x}{X} \cong L(I_x(X)/I_x^2(X),\mathbb{R})
      =\{F\colon I_x(X)/I_x^2(X)\to\mathbb{R}\mid F \text{ is linear}\,\}.
    \]
  \end{enumerate}
\end{prop}

\begin{proof}
  The first statement is proved in \cite[Proposition~3.11]{CW}. 
  The second statement follows by the same argument, omitting the smoothness assumption.
\end{proof}

\section{A diffeological space with non-smooth derivations}

In this section we present a concrete diffeological space for which 
the external and right tangent spaces do not coincide.  

\maintheoremone

\begin{proof}
  We first compute the right tangent space at the wedge point.
  For each $n$, denote by $\iota_n\colon\mathbb{R}_n\to X$ the canonical inclusion.

  By the universal property of the quotient diffeology, a germ 
  $[f]\in\germ(X,0)$ corresponds to a family of germs
  $([f_n])_{n\in\mathbb{N}}$ with
  $[f_n]\in\germ(\mathbb{R}_n,0)$ and $f_n(0)=f_m(0)$ for all $n,m$.
  Thus we obtain an isomorphism of $\mathbb{R}$-algebras
  \[
    \germ(X,0)\cong
    \Bigl\{(f_n)_n\in\prod_{n\in\mathbb{N}}\germ(\mathbb{R}_n,0)
      \Bigm| f_n(0)=f_m(0)\ \text{for all } n,m\in\mathbb{N}\Bigr\}.
  \]
  One readily checks that $I_0(X)$ corresponds to $\prod_n I_0(\mathbb{R}_n)$ and that
  $I_0(X)^2$ corresponds to $\prod_n I_0^2(\mathbb{R}_n)$. Hence we obtain
  \[
    I_0(X)/I_0^2(X)
    \cong
    \prod_{n\in\mathbb{N}} I_0(\mathbb{R}_n)/I_0^2(\mathbb{R}_n)\cong \prod_{n\in\mathbb{N}} T_0 (\mathbb{R}_n)\cong \prod_{n\in\mathbb{N}}\mathbb{R}.
  \]
  By \Cref{prop:external-right-Ix}\,\textup{(2)} we therefore have
  \[
    \righttangent{0}{X}
    \cong L\bigl(I_0(X)/I_0^2(X),\mathbb{R}\bigr)
    \cong L\Bigl(\prod_{n\in\mathbb{N}}\mathbb{R},\mathbb{R}\Bigr)
    = \Hom_{\mathbb{R}}\Bigl(\prod_{n\in\mathbb{N}}\mathbb{R},\mathbb{R}\Bigr).
  \]

  Next we compute the external tangent space at the wedge point.
  We write $[f]$ for the class of $f\in I_0(X)$ in $I_0(X)/I_0^2(X)$.  
  Writing out the isomorphism observed above, we have
  \[
  \Phi\colon I_0(X)/I_0^2(X)\to\prod_{n\in\mathbb{N}}\mathbb{R},
  \quad
  \Phi([f])
  =\Bigl(\left.\frac{d}{dt}\right|_{t=0}f(\iota_n(t))\Bigr)_{n\in\mathbb{N}},
  \]
  whose inverse is given by
  \[
  \Psi\bigl((a_n)_{n\in\mathbb{N}}\bigr)
  =[f],
  \qquad\text{where}\quad
  f(\iota_n(t))=a_n t.
  \]
  It is straightforward to check that $\Phi$ and $\Psi$ are well-defined and mutually inverse.  

  The maps $\Phi$ and $\Psi$ are in fact smooth as well.
  For $\Phi$, it suffices to show that each component $\Phi_n$ is smooth.
  Let $p:U\to I_0(X)/I_0^2(X)$ be a plot and fix $u_0\in U$.
  By the definitions of the diffeologies on $\germ(X,0)$ and on $I_0(X)/I_0^2(X)$,
  there exist a neighborhood $W\subset U$ of $u_0$, a $D$-open neighborhood $B\subset X$ of $0$,
  and a plot $F:W\to C^\infty(B,\mathbb{R})$ such that $p(u)=[F(u)]$ for $u\in W$.
  Then for $u\in W$,
  \[
  (\Phi_n\circ p)(u)
  =\left.\frac{\partial}{\partial t}\right|_{t=0}\ad(F)\bigl(u,\iota_n(t)\bigr).
  \]
  Set $B_n:=\iota_n^{-1}(B)$, an open neighborhood of $0$ in $\mathbb{R}_n$.
  Since $\ad(F):W\times B\to\mathbb{R}$ is smooth and $\id_W\times \iota_n|_{B_n}:W\times B_n\to W\times B$ is smooth,
  the composite $(u,t)\mapsto \ad(F)(u,\iota_n(t))$ is smooth near $(u_0,0)$.
  Hence $\Phi_n\circ p$ is smooth near $u_0$.
  Therefore $\Phi_n$ is smooth, and hence $\Phi$ is smooth.



  Conversely, for a plot $q\colon \plotdom{q}\to\prod_n\mathbb{R}$,
  define a smooth map
  \[
  \tilde q\colon \plotdom{q}\to C^{\infty}(X,\mathbb{R}),
  \qquad
  \tilde q(u)(\iota_n(t))=q_n(u)t.
  \]
  Consider the composite
  \[
  \plotdom{q}
  \xrightarrow{\;\tilde q\;}
  C^{\infty}(X,\mathbb{R})
  \longrightarrow
  \germ(X,0),
  \]
  where the second arrow is the canonical map to the germ algebra.  
  Its image lies in the ideal $I_0(X)\subset\germ(X,0)$,
  so we may regard this composite as a map $\plotdom{q}\to I_0(X)$.
  Then $\Psi\circ q$ is obtained by further composing this map with the quotient
  $I_0(X)\to I_0(X)/I_0^2(X)$.  
  Hence the composite $\Psi\circ q$ is a plot of $I_0(X)/I_0^2(X)$, and therefore $\Psi$ is smooth.

  Thus $\Phi$ and $\Psi$ are diffeomorphisms, and by \Cref{prop:external-right-Ix}, 
  we obtain an isomorphism
  \[
    \external{0}{X}\cong L^{\infty}\!\left(\prod_{n\in\mathbb{N}}\mathbb{R},\mathbb{R}\right).
  \]
  To finish the argument, we identify the space on the right.
  For each $i\in\mathbb{N}$, let $e_i\in \prod_{n\in\mathbb{N}}\mathbb{R}$ denote the vector whose $i$-th component is $1$ and all other components are $0$ (in particular $e_i\in\bigoplus_{n\in\mathbb{N}}\mathbb{R}$),
  and let $e_i^*\colon \prod_{n\in\mathbb{N}}\mathbb{R}\to\mathbb{R}$ be the $i$-th coordinate projection.
  The theorem follows from the next \Cref{prop:smooth-linear-finite-support}.
\end{proof}

\begin{prop}\label{prop:smooth-linear-finite-support}
  The canonical linear map
  \[
    \bigoplus_{n\in\mathbb{N}}\mathbb{R} \longrightarrow
    L^{\infty}\!\left(\prod_{n\in\mathbb{N}}\mathbb{R},\mathbb{R}\right);\quad
    (a_i)_{i}\longmapsto \sum_{i\in\mathbb{N}} a_i e_i^*,
  \]
  is an isomorphism of vector spaces. In particular, every smooth linear map
  $D\colon \prod_{n\in\mathbb{N}}\mathbb{R}\to\mathbb{R}$ is a finite linear combination of the
  coordinate projections.
\end{prop}

\begin{proof}
Since the injectivity is straightforward, we show that this map is surjective.
Let $D\in L^{\infty}\!\left(\prod_{n\in\mathbb N}\mathbb R,\mathbb R\right)$ be a smooth linear map.
For each $N\in\mathbb N$, let
\[
  \pi_N\colon \prod_{n\in\mathbb N}\mathbb R \longrightarrow \mathbb R^N
\]
be the projection onto the first $N$ coordinates.
We claim that there exists $N\in\mathbb N$ such that
\[
  D(x)=0 \qquad \text{for all } x\in\ker(\pi_N)
  =\{x\in\prod_{n\in\mathbb N}\mathbb R \mid x_1=\cdots=x_N=0\}.
\]
Suppose otherwise.
Then for every $N\in\mathbb N$ there exists $x_N\in\ker(\pi_N)$ with $D(x_N)\neq 0$.
Replacing $x_N$ by a scalar multiple, we may assume $D(x_N)=N$.
Set $U=(-1,1)\subset\mathbb R$.
Define
\[
  t_k=2^{-3k},\qquad
  I_k=\bigl(2^{-3k-1},\,2^{-3k+1}\bigr)\subset U
  \qquad (k\in\mathbb N).
\]
Then $t_k\to 0$ and the intervals $I_k$ are pairwise disjoint with $t_k\in I_k$.
For each $k\in\mathbb N$, choose a smooth function $\varphi_k\colon U\to\mathbb R$ such that
\[
  \operatorname{supp}(\varphi_k)\subset I_k
  \qquad\text{and}\qquad
  \varphi_k(t_k)=1.
\]
Define a map $p\colon U\to\prod_{n\in\mathbb N}\mathbb R$ by
\[
  (p(u))_n = \sum_{k=1}^{n} (x_k)_n\,\varphi_k(u)
  \qquad (n\in\mathbb N).
\]
For each fixed $n$, the sum is finite, hence the component $u\mapsto (p(u))_n$ is smooth.
Therefore $p$ is a plot of $\prod_{n\in\mathbb N}\mathbb R$ equipped with the product diffeology.

For each $k\in\mathbb N$, we have $p(t_k)=x_k$.
Indeed, since the supports of the $\varphi_\ell$ are pairwise disjoint,
we have $\varphi_\ell(t_k)=0$ for $\ell\neq k$, so
\[
  (p(t_k))_n =
  \begin{cases}
    (x_k)_n & \text{if } k\leq n,\\
    0       & \text{if } k>n.
  \end{cases}
\]
If $k>n$, then $(x_k)_n=0$ since $x_k\in\ker(\pi_k)$.
Hence $(p(t_k))_n=(x_k)_n$ for all $n\in\mathbb N$, and thus $p(t_k)=x_k$.
Consequently,
\[
  (D\circ p)(t_k)=D(p(t_k))=D(x_k)=k.
\]
Since $t_k\to 0$ in $U$ while $(D\circ p)(t_k)=k\to\infty$,
the map $D\circ p\colon U\to\mathbb R$ is not continuous at $0$.
This contradicts the smoothness of $D\circ p$.
Therefore the claim holds for some $N\in\mathbb N$.

Fix such an $N$.
If $\pi_N(x)=\pi_N(y)$, then $x-y\in\ker(\pi_N)$ and hence $D(x)=D(y)$.
Thus $D$ factors through $\pi_N$ and there exists a unique linear map
$\bar D\colon \mathbb R^N\to\mathbb R$ such that $D=\bar D\circ \pi_N$.
Since $\mathbb R^N$ is finite-dimensional, there exist $c_1,\dots,c_N\in\mathbb R$ such that
\[
  \bar D(u_1,\dots,u_N)=\sum_{i=1}^N c_i u_i.
\]
Since $x-(x_1,\dots,x_N,0,\dots)\in\ker(\pi_N)$, we have
\[
  D(x)=D(x_1,\dots,x_N,0,\dots)=\sum_{i=1}^N c_i x_i=\sum_{i=1}^N c_i e_i^*(x).
\]
Thus $D=\sum_{i=1}^N c_i e_i^*$, which proves surjectivity.
\end{proof}

    

\begin{remark}\label{rem:prodRn}
  Let $Y=\prod_{n\in\mathbb N}\mathbb R$ be equipped with the product diffeology.
  Motivated by \Cref{prop:smooth-linear-finite-support}, we expect an isomorphism of
  diffeological vector spaces
  \[
    I_0(Y)/I_0(Y)^2 \;\cong\; \bigoplus_{n\in\mathbb N}\mathbb R,
  \]
  where the right-hand side carries the fine diffeology; see \cite[3.7]{IZ} for its
  definition and basic properties.
  If this holds, then both $\hat T_0(Y)$ and $\hat T^R_0(Y)$ identify with
  \[
    L^{\infty}\!\left(\bigoplus_{n\in\mathbb N}\mathbb R, \mathbb{R}\right)
    \;=\;
    L\!\left(\bigoplus_{n\in\mathbb N}\mathbb R, \mathbb{R}\right)
    \;\cong\;
    \prod_{n\in\mathbb N}\mathbb R,
  \]
  and in particular $\hat T_0(Y)=\hat T^R_0(Y)$.
  We do not include a proof, since it is not used in this paper and would require a separate and detailed analysis. 
\end{remark}

\section*{Acknowledgments}
    I would like to express my sincere gratitude to my supervisor, Takuya Sakasai, for his valuable guidance and continuous support throughout this research. 
    I am also grateful to Katsuhiko Kuribayashi for many insightful discussions and helpful advice. 
    I thank Toshiyuki Kobayashi for his generous support and encouragement as my supporting supervisor in the WINGS-FMSP program.  
    This work was supported by JSPS Research Fellowships for Young Scientists and KAKENHI Grant Number JP24KJ0881. 
    Lastly, I would like to acknowledge the WINGS-FMSP program for its financial support.

\bibliographystyle{plain}
\bibliography{nonsmoothop}



\end{document}